\documentclass[preprint,12pt]{elsarticle}
\usepackage{amscd,amssymb,stmaryrd}
\usepackage{amsmath}
\usepackage{graphicx}
\usepackage{subcaption}
\usepackage[utf8]{inputenc}
\usepackage[export]{adjustbox}
\usepackage{wrapfig}
\usepackage{amstext}
\usepackage{pstricks,pst-node,pst-plot,pst-coil}
\usepackage{amsthm}
\usepackage{amsmath}
\usepackage{color}
\usepackage{mathtools}
\usepackage{hyperref}
\usepackage[english]{babel}
\usepackage{booktabs}
\usepackage{mathrsfs}
\usepackage{comment}
\usepackage{float}
\usepackage[linesnumbered,ruled,vlined, ruled,vlined]{algorithm2e}

\theoremstyle{definition}
\newtheorem{theorem}{Theorem}[section]

\theoremstyle{remark}

\begin{document}
\begin{frontmatter}
\title{Multi-variance replica exchange stochastic gradient MCMC for inverse and forward Bayesian physics-informed neural network}
\author[a,b]{Guang Lin}
\author[c]{Yating Wang}
\author[a]{Zecheng Zhang}
\cortext[cor1]{Corresponding author}
\address[a]{Department of Mathematics, Purdue University, West Lafayette, IN 47907, USA}
\address[b]{Department of Mechanical Engineering, Department of Statistics (courtesy), Purdue University, West Lafayette, IN 47907, USA}
\address[c]{Department of Mathematics, The University of Hong Kong, Pokfulam Road, Hong Kong SAR, China.}

\begin{abstract}
    
    Physics-informed neural network (PINN) has been successfully applied in solving a variety of nonlinear non-convex forward and inverse problems. However, the training is challenging because of the \textcolor{black}{non-convex} loss functions and the multiple optima in the Bayesian inverse problem. 
    In this work, we propose a multi-variance replica exchange stochastic gradient Langevin diffusion method \textcolor{black}{to tackle the challenge of the multiple local optima in the optimization and the challenge of the multiple modal posterior distribution in the inverse problem}.
     Replica exchange methods are capable of escaping from the local traps and \textcolor{black}{accelerating the convergence};
     two chains with different temperatures are designed where the low temperature chain aims for the local convergence, and the target of the high temperature chain is to travel globally and explore the whole \textcolor{black}{loss function entropy landscape}.
    However, it may not be efficient to solve mathematical inversion problems by using the vanilla replica method directly since the method doubles the computational cost in evaluating the forward solvers (likelihood functions) in the two chains. To address this issue, we propose to make different assumptions on the energy function estimation and this facilities one to use \textcolor{black}{solvers} of different fidelities in
    the likelihood function evaluation. \textcolor{black}{More precisely, one can use a solver with low fidelity in the high temperature chain while use a solver with high fidelity in the low temperature chain.}
    Our proposed method significantly lowers the computational cost 
    in the high temperature chain, meanwhile preserves the accuracy and converges very fast.
    We give an unbiased estimate of the swapping rate and give an estimation of the discretization error of the scheme.
    To verify our idea,  
    we design and solve four inverse problems which have multiple modes. 
    The proposed method is also employed to train the Bayesian PINN to solve the forward and inverse problems; faster and more accurate convergence has been observed when compared to \textcolor{black}{the stochastic gradient Langevin diffusion (SGLD) method and vanila replica exchange methods.}
\end{abstract}


\end{frontmatter}

\section{Introduction}
	The development of deep neural networks (DNNs) has drawn people's attention to incorporate Bayesian approach into the learning of DNNs \cite{ahn2012bayesian, sgld, ding2014bayesian}. Stochastic gradient Markov chain Monte Carlo methods \cite{sgld, PSGLD,SGRLD, ma2015complete,chen2014stochastic, simsekli2016stochastic} have become attractive due to its potential for uncertainty quantification applications. In this family of approaches, the stochastic gradient Langevin dynamics \cite{sgld} algorithm was first proposed which bridges the gap between optimization and Bayesian inference.
	It was originates from the discretization of Langevin diffusion, but uses small batches to approximate the true gradients. The updating formula of SGLD resembles that of stochastic gradient descent (SGD), thus it can be naturally used in deep learning tasks with large datasets and is easy to implement in practice. By injecting a judicious choice of 
	noise when updating the DNN parameters, the method ensures that the generated samples will converge to the true posterior instead of the MAP \cite{sgld, vollmer2016exploration, sg-mcmc-convergence}. 
	
	Many variants of stochastic gradient MCMC methods have also been proposed to improve the convergence or stability of SGLD. In the case when the parameters of the DNN have complicated posterior density \cite{dauphin2014identifying, PSGLD, chen2014stochastic} such as different scales of variance or highly correlated, some preconditioning methods were introduced taking into account the geometric information. These include the preconditioned SGLD, Hessian approximated SGMCMC \cite{PSGLD, simsekli2016stochastic, ma2015complete, psgld-sa, hessian-sgld-sa}, where some preconditioners are proposed to approximation the Hessian information of the log posterior. Other improvements include the high-order integrators in the discretization of Langevin diffusion \cite{chen2015convergence}, for example, stochastic gradient Hamiltonian Monte Carlo \cite{chen2014stochastic}, stochastic gradient thermostats \cite{ding2014bayesian}. To combine simulated tempering with the traditional MCMC community, a replica stochastic gradient MCMC (reSG-MCMC) was recently brought up \cite{deng2020non}. 
	
	The main idea of the standard replica exchange Langevin dynamics is to use multiple diffusion processes with different temperatures. The low temperature diffusion process can exploit the local area, and the high temperature one will explore the global domain. With a suitable swapping criteria, the replica exchange method admits a trade-off between global exploration and local exploitation, which accelerates the convergence rate of the samples to the invariant correct distribution. It is feasible and can be naturally extended for parallel computing. However, it is important to propose a good swapping scheme which aims to overcome the issue of large bias introduced by using mini-batches. In \cite{deng2020non}, the authors proposed 
	an adaptive replica exchange SG-MCMC method with the help of stochastic approximation, which has a correction term to reduce the bias from the swaps based on theoretical analysis. 
	
	Compared to the other MCMC sampling methods, the replica exchange methods have the fast convergence property. 
	However, it usually requires a higher computational cost due to the replica computation high temperature chain. To be specific, take as an example
	solving inverse physical problems under the Bayesian framework with efficient sampling algorithms, which has been a favored approach in the last century \cite{stuart2010inverse}. 
	In the Bayesian approach, one treats the network weight parameters as random variables and infers them conditioned on limited data. The likelihood function consists of the loss between the the network output and observation data. 
	For most of the physical problems, the computation of the loss involves the evaluation of the forward solver which maps the inverse target to the observations. 
	Due to the mathematical properties of the equations, such as problems with multiscale nature, this can be a very time consuming process \cite{efendiev2006preconditioning, chung2020multi}. Hence, there should be a balance between the fast convergence and high computational cost, and a modification to accelerate the computation of two likelihood (energy) functions in the two chains is needed.
	
	Our idea is to overcome the difficulty by applying a coarser solver in the high temperature chain to explore the model parameter space. The motivation is that the high temperature chain is designed to guide the particle travel globally and escape from the local optimum trap, and it is practical to use a computational efficient coarse solver with little sacrifice on the accuracy. 
    Since the errors in the forward solver and the other sources of errors contribute to the energy function estimation, one needs to make different assumptions on the energy functions in the two chains.
    To the extent of our knowledge, this work is the first to study the theoretical and practical applications based on this idea. 
	
	The key contributions of this work are
\begin{itemize}
\item We introduce a multi-variance replica exchange method for Bayesian physics-informed neural networks. Adopting different variances for the energy function estimators in different diffusion processes allows us to take advantage of different model fidelities to perform numerical simulations.

\item We derive a correction term to ensure the reduction of bias in the swapping scheme and prove the discretization error of the method. 

\item The proposed multi-variance replica exchange method is feasible in general Bayesian learning tasks. In particular, we use it in training the Bayesian PINN (physics-informed neural network).
\textcolor{black}{More precisely, physical laws are integrated in the deep neural network by the constraints in the loss; 
and we use noisy training data, and the outcome of the network can provide uncertainties of the prediction.}

\item The proposed method is applied to solve both forward and inverse PDE problems. The numerical results demonstrate the improved  efficiency in capturing both multiple modes and infinite modes in some inverse problems. It also shows \textcolor{black}{the better predictions and faster convergence} with reliable uncertainties in solving PINN problems \textcolor{black}{when compared to the classical gradient based methods}.

\end{itemize}	
	
	The outline of the paper is as follows. In Section \ref{sec:prelim}, we review some backgrounds in stochastic gradient Langevin dynamics, replica exchange stochastic gradient MCMC algorithms, and Bayesian framework for inverse problems \ref{sec:bayes}. Our main algorithm is presented in Section \ref{sec:main_method}. Numerous numerical examples are shown to illustrate the efficiency of proposed method in Section \ref{sec:numerical}.

\section{Preliminaries}\label{sec:prelim}
\subsection{Replica exchange Langevin diffusion (reLD)}
Diffusion-based Markov Chain Monte Carlo (MCMC) algorithms have become increasingly popular in the
recent years due to the real applications performance and theoretic supports \cite{nguyen2019non,csimcsekli2017fractional, simsekli2020fractional,raginsky2017non}. One of the most popular choices is the Langevin diffusion which is a stochastic differential equation as follows:
\begin{align}
    d\beta_t = -\nabla U(\beta_t)dt+\sqrt{2\tau}dW_t.
    \label{bg_lg}
\end{align}
Here $\beta_t\in\mathbb{R}^d$, $\tau\in\mathbb{R}^{+}$ is called temperature; $W_t$ is the standard Brownian motion in $\mathbb{R}^d$.
$U(\cdot)$ is the energy function which is often convex.
Denote by $D = \{d_i\}_{i= 1}^{N} $ the train data, where $d_i$ is a data point, one choice of the energy function is:
\begin{align}
    U(\beta) = -log(p(\beta))-\sum_{i = 1}^{N}log (p(d_i|\beta)), 
\end{align}
where $\beta\in\mathbb{R}^d$ and $p(\beta)$ is the prior and $p(d_i|\beta)$ is the likelihood.
People have established that under mild conditions, the solution process $\{\beta_t\}_{t\geq 0}$ will converge to
the invariant distribution
$\pi(\beta_t):= e^{-U(\beta_t)/\tau}$ \cite{roberts2002langevin, nguyen2019non}. 
However, when the temperature $\tau$ is large, the flattened distribution is less concentrated around the global optimum
and the geometric connection to the global
minimum is affected \cite{raginsky2017non}. 
As a result, the particle is exploring the whole loss function landscape and can travel globally \cite{zhang2017hitting}.
 On the other hand, a low-temperature results in the exploitation:
 the particle is exploiting the local region; the local optima can be reached quickly; however, the particle is hard to escape from the local trap (local optima) which is harmful for the global convergence \cite{deng2020non}.

This motivates people designing the mixed algorithm; the replica exchange Langevin diffusion (reLD) is then designed.
In this algorithm, two chains with different temperatures are used,
\begin{align*}
    &d\beta_t^{1} = -\nabla U(\beta_t^{1})dt+\sqrt{2\tau_1}dW_t^1\\
    &d\beta_t^{2} = -\nabla U(\beta_t^{2})dt+\sqrt{2\tau_2}dW_t^2.
\end{align*}
People can switch the particles $(\beta^1_t, \beta^2_t)$ between two chains. 
More precisely, 
let $r>0$ be a constant and if one follows the 
swapping rate ,
$rS(\beta^1_t, \beta^2_t)dt$, where,
\begin{align}
    S(\beta^1_t, \beta^2_t) = e^{(\frac{1}{\tau_1}-\frac{1}{\tau_2})(U(\beta^1_t)-U(\beta^2_t))},
    \label{bg_swap_rate}
\end{align}
the system then converges to the invariant distribution; the density is given as follow:
\begin{align}
    \pi(\beta^1, \beta^2) \sim exp(-\frac{U(\beta^1)}{\tau_1}-\frac{U(\beta^2)}{\tau_2}).
    \label{bg_inv_dis}
\end{align}

\subsection{Stochastic Gradient
Langevin Dynamics (SGLD)}
The MCMC aalgorithm can can be used as an optimization algorithm. Consider the discretization of (\ref{bg_lg}),
\begin{align}
    b_{k+1}^{0} = b_{k}^{0}-\eta_k\nabla U(b_k^{0})+\sqrt{2\eta_k\tau} \varepsilon_k, 
    \label{bg_sg_disc}
\end{align}
where $\{b_k^{0}\}_{k = 1,..., N}\subset \mathbb{R}^d$ solves the equation; $\eta_k>0$ is a constant and can be understood as the learning rate; $\varepsilon_k$ is a standard d-dimensional Gaussian vector. Equation (\ref{bg_sg_disc}) is a stochastic gradient descent (SGD) algorithm with some perturbations, the discretized equation returns to the gradient descent algorithm when $\tau = 0$. 
The connection between
Langevin MCMC and the gradient descent has been recently established in \cite{dalalyan2017further, simsekli2020fractional} for strongly convex $U(\cdot)$. Moreover, \cite{raginsky2017non, xu2017global} proved non-asymptotic guarantees for this perturbed scheme.
In real applications, $N$ is large and the likelihood function is hard to evaluate \cite{efendiev2006preconditioning, chung2020multi}, this brings difficult in calculating the energy function; hence the approximation is often used and this will result in an error in the convergence.
Furthermore, in physical applications, calculating the likelihood usually involves evaluating the forward solver, which also contributes an error in the convergence.
In the theoretical analysis, if we denote $\hat{U}(\cdot)$ as the estimator of the energy function, people usually assume that
\begin{align}
    \hat{U}(\beta)\sim N(U(\beta), \sigma^2),
    \label{bg_energy_assumption}
\end{align}
where $\sigma$ is 
the variance and is the reflection of the error.

Motivated by the stochastic gradient descent algorithm (SGD), people can build an estimator of the energy function with a random subset of samples.
More precisely,
\begin{align}
    \hat{U}(\beta) = -log(p(\beta))-\frac{N}{n}\sum_{i = 1}^{N}log (p(d_{s_i}|\beta)), 
    \label{bg_approx_energy}
\end{align}
where $\{s_i\}_{i = 1}^n$ is the index of a batch of samples. We hence have the stochastic gradient Langevin diffusion (SGLD),
\begin{align}
    b_{k+1}^{SGLD} = b_{k}^{SGLD}-\eta_k\nabla \hat{U}(b_k^{SGLD})+\sqrt{2\eta_k\tau} \varepsilon_k.
\end{align}
In \cite{teh2016consistency}, the authors have shown that SGLD
asymptotically converges to a unique invariant distribution. 

\subsection{Replica Exchange Stochastic Gradient
Langevin Dynamics (reSGLD)}
In the previous section, we have introduced the reLD and the optimization algorithm constructed using the Langevin diffusion.
It is then natural to use the reLD to build an optimization algorithm, the reSGLD is then proposed. The discretization is the same 
as the SGLD and we have,
\begin{align}
    &b_{k+1}^{re, 1} = b_{k}^{re, 1}-\eta_k\nabla \hat{U}(b_k^{re, 1})+\sqrt{2\eta_k\tau_1} \varepsilon_k, \\
    &b_{k+1}^{re, 2} = b_{k}^{re, 2}-\eta_k\nabla \hat{U}(b_k^{re, 2})+\sqrt{2\eta_k\tau_2} \varepsilon_k, 
    \label{bg_disc}
\end{align}
where $\tau_1<\tau_2$ are the temperatures, $(b_k^{re, 1}, b_k^{re, 2})$ are the corresponding solutions. 
By following the estimation (\ref{bg_energy_assumption}) and allowing swapping, i.e., $(b_k^{re, 1}, b_k^{re, 2}) = (b_{k+1}^{re, 2}, b_{k+1}^{re, 2})$
with a rate which is an unbiased estimation\footnote{It should be noted that the original swapping rate (\ref{bg_swap_rate}) becomes bias due to the stochastic approximation of the energy function.} of (\ref{bg_swap_rate}):
\begin{align*}
    \hat{S}(b^1, b^2) = exp(\frac{1}{\tau_1}-\frac{1}{\tau_2})(\hat{U}(b^1)-\hat{U}(b^2) -(\frac{1}{\tau_1}-\frac{1}{\tau_2})\sigma^2 ),
\end{align*}
The authors theoretically give the discretization error of the scheme and the numerical results also surpass the benchmarks of several datasets.

\section{Bayesian Framework of the Inversion Problems}\label{sec:bayes}
One can solve the inverse problem using the Bayesian framework.
A favored approach is to sample the inverse target by some MCMC algorithm. Let the observation of the problem be $y$ and the quantity of interest in the inverse problem be $k$. 
The target is to generate samples that approximate the posterior $P(k|y)$. By the Bayes' Theorem, 
$$
P(k|y) \propto P(y|k)\cdot P(k),
$$
where $P(k)$ is the prior and $P(y|k)$ is the likelihood.
The key in the implementation is to design and evaluate the likelihood function.
This usually involves solving an forward problem given the proposal $k$ and measures the difference between the observation and the solution of the forward problem.
In addition to the error in the forward solver, there are other sources of errors \cite{efendiev2006preconditioning, li2019diffusion} such as the errors in the real observation.
One can assume the combined error follows a normal distribution with standard deviation $\sigma$, it follows that,
\begin{align}
    P(y|k) \propto \exp \left \{ {-\frac{\|y-\mathcal{F}(k)\|_2^2}{\sigma^2}} \right \},
\end{align}
where $\mathcal{F}$ denotes the forward solver and the $\sigma$ can be viewed as the precision of the observations and the solver.
There are two difficulties in practice. 
The first one is how to generate the proposals. This requires the design of the transition probability and one usually seeks ideas from the MCMC literature.
One approach is to use the reinforcement learning to learn a transition probability function \cite{chung2020multi}.
In this work, we are going to propose samples under the framework of Langevin dynamics.
Another difficulty is the choice of forward solver.
The simulation process is usually slow in particular if one requires high accuracy.
Though Replica exchange method accelerates convergence in terms of the MCMC procedure, one still needs to run the forward solver several times due to its multi-chain property. This greatly increases 
the computational cost.
Because of this limitation, the replica exchange method may not be directly applied to solve physical inverse problems.
In this work, we propose a method which greatly reduces the computation time, and meanwhile preserves the fast convergence property of the replica exchange method. 

\section{Proposed Method} \label{sec:main_method}
In this section we are going to present the proposed method. 
Replica exchange methods have been proved to capture the global optima successfully; however, it may be slow since we need to evaluate the likelihood functions twice in the low and high temperature chains separately. The problem becomes exacerbated in solving inversion physical problems in the Bayesian framework. This is because people need to evaluate an forward solver in calculating the likelihood functions, which is extremely time consuming; 
if we repeat the process twice, this will largely increase the computation time.
People hence need to optimize the replica method to use it in real world problems.

The question is: do we need to use the same solver in calculating the energy functions in two chains? Our proposal is to use a fine solver in the chain with low temperature while applying a coarser solver in the high temperature chain which travels globally.
The motivation of the proposed method is: the high temperature chain is exploring the region, its job is to get rid of the local optimum and find the potential region where the global optima lays in; it hence requires less accuracy. For the better convergence to the global optima, once the particle is closed to the region, we use the low temperature chain which has a fine solver for the better accuracy.
Basing on this reasoning, we then can use a coarse solver in the high temperature chain and assume a larger variance in the energy estimation. Our experiments show that the accuracy will not be compromised and we prove that the discretization error will not be changed.
We hence can benefit from the computation cost saved in the coarse solver evaluation.

As we have discussed in the last section, 
the solver contributes errors in the energy functions; together with error in the stochastic gradients 
one may assume that the estimations satisfies, for example, (\ref{bg_energy_assumption}). 
We propose to use different solvers; this means that the energy function estimations are different, hence a new swapping rate should be derived.
Now let $(b^1_k, b^2_k)$ denote the solution of the proposed scheme, the scheme can be summarized as follow, 
\begin{align}
    &b_{k+1}^{1} = b_{k}^{1}-\eta_k\nabla \hat{U}_1(b_k^{1})+\sqrt{2\eta_k\tau_1} \varepsilon_k, \\
    &b_{k+1}^{2} = b_{k}^{2}-\eta_k\nabla \hat{U}_2(b_k^{2})+\sqrt{2\eta_k\tau_2} \varepsilon_k, 
    \label{pp_dis}
\end{align}
where $\hat{U}_1(b_k^1)\sim N(U(b_k^1), \sigma_1^2)$ and $\hat{U}_2(b_k^2)\sim N(U(b_k^2), \sigma_2^2)$ are the two estimators. We denote $(\beta^1, \beta^2)$ as the continuous solution with the different estimator,
then the new estimator is summarized in the following theorem.

\begin{theorem}
Let $a_1, a_2>0$ and satisfies $a_1+a_2 = 1$; set $b_i =  a_i^2$ 
for $i = 1,2$ and denote $\tau_\delta = \frac{1}{\tau_1}-\frac{1}{\tau_2}$ .
Then the estimator:
\begin{align}
   \Hat{S} = e^{\tau_\delta\bigg(a_1\big(\Hat{U}_1(\beta^1)-\Hat{U}_1(\beta^2)\big)-a_2\big(\Hat{U}_2(\beta^1)-\hat{U}_2(\beta^2) \big)
    - (a_1\sigma_1  + a_2\sigma_2)^2\tau_\delta
     \bigg) }
     \label{pp_rate}
\end{align}
is unbiased.
\end{theorem}
\begin{proof}
By the assumption on $\Hat{U}_1$ and $\Hat{U}_2$ we have,
\begin{align*}
    & \Hat{U}_1(\beta^1)-\Hat{U}_1(\beta^2) = 
    U(\beta^1)-U(\beta^2)+\sqrt{2}\sigma_1W_1,\\
    & \Hat{U}_2(\beta^1)-\Hat{U}_2(\beta^2) = 
    U(\beta^1)-U(\beta^2)+\sqrt{2}\sigma_2W_1
\end{align*}
\textcolor{black}{where $W_1$ can be taken as a Brownian motion when $t = 1$ which is a Gaussian vector}. The switching rate $\Hat{S}_t$ for $t\in[0, 1]$ is then given and simplified as follow:
\begin{align*}
    \Hat{S}_t &= exp\bigg(\tau_\delta\bigg[
    a_1 \big( \Hat{U}_1(\beta^1)-\Hat{U}_1(\beta^2) \big)+
    a_2 \big( \Hat{U}_2(\beta^1)-\Hat{U}_2(\beta^2) \big)\\
    &-\tau_\delta a_1^2\sigma_1^2 t -\tau_\delta a_2^2\sigma_2^2 t-
    2\tau_\delta a_1a_2\sigma_1\sigma_2 t \bigg]\bigg)\\
    &= 
    exp\bigg(\tau_\delta\bigg[U(\beta^1)-U(\beta^2)
    -\tau_\delta a_1^2\sigma_1^2 t -\tau_\delta a_2^2\sigma_2^2 t-
    2\tau_\delta a_1a_2\sigma_1\sigma_2t\\
    &+(\sqrt{2}a_1\sigma_1+\sqrt{2}a_2\sigma_2) W_t \bigg] \bigg)
\end{align*}
The partial derivatives of $\Hat{S}_t$ are straightforward as follow,
\begin{align*}
    &\frac{d \Hat{S}_t}{dt } = -\tau_\delta^2(a_1^2\sigma_1^2+a_2^2\sigma_2^2+2a_1a_2\sigma_1\sigma_2)\Hat{S}_t,\\
    &\frac{d \Hat{S}_t}{d W_t } = 
    \sqrt{2}\tau_\delta(a_1\sigma_1+a_2\sigma_2)\Hat{S}_t,\\
    &\frac{d^2 \Hat{S}_t}{d W_t^2 } = 
    2\tau_\delta^2(a_1\sigma_1+a_2\sigma_2)^2\Hat{S}_t.
\end{align*}
Ito's lemma then reads,
\begin{align*}
    d \Hat{S}_t = \sqrt{2}\tau_\delta(a_1\sigma_1+a_2\sigma_2)\Hat{S}_tdW_t.
\end{align*}
Thus $\{\Hat{S}_t\}_{t\in[0,1]}$ is Martingale. Set $t = 1$, the desired estimator is unbiased.
\end{proof}

\subsection{Discretization Error}
In this section, we comment on the discretization error of the proposed method.
We first introduce the strong form solution $\beta_t$ which satisfies:
\begin{align}
    d\beta_t = -\nabla G(\beta_t) dt+\Sigma_tdW_t.
    \label{strong_form}
\end{align}
Here 
$G = \big[ U(\beta_t^1); U(\beta_t^2) \big]$, $\beta_t^1, \beta_t^2\in\mathbb{R}^d$ and $U(\cdot)$ is the energy function. $W_t\in\mathbb{R}^{2d}$ is a Brownian motion and $\Sigma_t$ is a stochastic process which switches between 
$\begin{pmatrix}
  \sqrt{2\tau_1 I_d} & 0\\ 
  0 & \sqrt{2\tau_2 I_d}
\end{pmatrix}$ and
$\begin{pmatrix}
  \sqrt{2\tau_2 I_d} & 0\\ 
  0 & \sqrt{2\tau_1 I_d}
\end{pmatrix}$
with the probability $rS(\beta^1_t, \beta^2_t)dt$; here $I_d\in\mathbb{R}^{d\times d}$ denotes the identity matrix and $S(\cdot, \cdot)$ is given in (\ref{bg_swap_rate}).

Let us denote $\beta^{\eta}(k) = [\beta^{\eta, 1}(k); \beta^{\eta, 2}(k)]$ as the solution of the full discretization scheme; that is, $\beta^{\eta}(k)$ satisfies,
\begin{align}
    \beta^{\eta}(k+1) = \beta^{\eta}(k)-\eta\nabla \hat{G}_1(\beta^{\eta}(k))dt+\sqrt{\eta}\hat{\Sigma}^{\eta}(k)\xi_k\
    \label{full_discretization}
\end{align}
where $\eta$ is a fix time step and $\xi_k\in\mathbb{R}^{2d}$ is a standard Gaussian distribution. $\hat{G}(\beta^{\eta}_t)$ is a stochastic process and switches between 
$\hat{G}_1(\beta^{\eta}(k)) = \big[ \Hat{U}_1\big(\beta^{\eta, 1}(k)\big); \Hat{U}_2\big(\beta^{\eta, 2}(k)\big) \big]$ and 
$\hat{G}_2 = \big[ \Hat{U}_1\big(\beta^{\eta, 2}(k)\big); \Hat{U}_2\big(\beta^{\eta, 1}(k)\big) \big]$ with the probability $r\hat{S}(\beta^{\eta, 1}(k), \beta^{\eta, 2}(k))\eta$. $\hat{\Sigma}(t)$ is defined the same as $\Sigma_t$ but is switching with the $r\hat{S}(\beta^{\eta, 1}(k), \beta^{\eta, 2}(k))\eta$.
We also define the semi-discretized solution $\{\beta_t^{\eta}\}_{t\geq 0}$ as the linear interpolation of $\{\beta^{\eta}(k)\}_{k\geq 1}$ which satisfies,
\begin{align}
    \beta_t^{\eta} = \beta_0^{\eta} - \int_0^t\nabla \hat{G}(\beta_{[s/\eta]\eta}^{\eta})ds
    +\int_0^t\hat{\Sigma}^{\eta}_{[s/\eta]\eta} dW_s.
    \label{semi_discretization}
\end{align}
Here $[\cdot]$ denotes the floor function and $\hat{\Sigma}^{\eta}_{[s/\eta]\eta}$ follows the same trajectory as $\hat{\Sigma}^{\eta}(s/\eta)$.
To prove the discretization theorem. Two assumptions on the energy functions should be made. We first assume that the energy function is $\alpha-$ smooth, more precisely, there exists $\alpha>0$ such that for every $p, q\in\mathbb{R}^d$, we have:
\begin{align}
    \|\nabla U(p)-\nabla U(q)\|\leq \alpha\|p-q\|. \label{assumption1}
\end{align}
We also assume the dissipativity of the energy function. That is, there are constants $a>0$ and $b\geq 0$ such that 
\begin{align}
    (p, \nabla U(p))\geq a\|p\|^2-b,\label{assumption2}
\end{align}
true for all $p\in\mathbb{R}^d$ and $(\cdot, \cdot)$ denotes the inner product. We also quantify the errors in the swapping rate and energy function estimator; that is, 
we have $\Hat{S}(\beta^1_k, \beta^2_k) = S(\beta^1_k, \beta^2_k)+\Phi^1_k$ and $\nabla \hat{U}_k = \nabla U_k+\Phi^2_k$, where $\Phi^1_k$ and $\Phi^2_k$ are in $\mathbb{R}^d$ and 
are \textcolor{black}{the functions measuring the error in the swapping rate and energy function estimations at $k-th$ iteration.}
\begin{theorem}
If we have the smoothness and dissipative assumptions, and we further assume that $\eta$ satisfies $0<\eta<1\cap \frac{a}{\alpha^2}$,
then there exist constants $C_1, C_2, C_3$ such that,
\begin{align}
    \mathbb{E}[\sup_{0\leq t\leq T}\|\beta_t-\beta_t^{\eta}\|^2]
    \leq C_1\eta+C_2\max_{k}\mathbb{E}[\|\Phi^2_k\|^2]+C_3\max_k\sqrt{\mathbb{E}[|\Phi^1_k|^2]},
\end{align}
\end{theorem}
where $C_1$ depends on the temperatures $\tau_1, \tau_2$, the dimension $d$ of $\beta_t$, \textcolor{black}{the finite time $T$} and the other constants $\alpha, a, b$ in the assumptions (\ref{assumption1}) and (\ref{assumption2}); and $C_2$ depends on $T$ and $\alpha$; $C_3$ depends on \textcolor{black}{$r$ which defines the swapping rate, $d, T$ and $\alpha$}.
We give an outline of the proof and please refer to \cite{deng2020non, chen2020accelerating} for the details.
\begin{proof}
Integrate (\ref{strong_form}) from $0$ to $t$ and substract (\ref{semi_discretization}), if we assume that $\beta_t(0) = \beta_0^{\eta}$, we then have,
\begin{align}
    \beta_t-\beta_t^{\eta} = - \int_0^t\bigg(\nabla G(\beta_s) -\hat{G}(\beta_{[s/\eta]\eta}^{\eta}) \bigg)ds
    +\int_0^t\bigg(\Sigma_s-\hat{\Sigma}^{\eta}_{[s/\eta]\eta}\bigg) dW_s, \text{ for } \forall t\in[0, T]
\end{align}
where $[\cdot]$ denotes the floor function.
Take expectation over $\beta_t$, it then follows that:
\begin{align}
    \mathbb{E}[\sup_{0\leq t \leq T}\|\beta_t-\beta_{t}^{\eta}\|^2  ]\leq
    2T\underbrace{\mathbb{E}\bigg[
    \int_0^T\|\nabla G(\beta_s)-\nabla\hat{G}(\beta_{[s/\eta]\eta}^{\eta})\|^2ds
\bigg]}_{\mathcal{I}}+8\underbrace{\mathbb{E}\bigg[
    \int_0^T \|\Sigma_s-\hat{\Sigma}^{\eta}_{[s/\eta]\eta}\|^2ds
    \bigg]}_{\mathcal{J}},
\end{align}
note that the above inequality is estimated due to the Burkholder-Davis-Gundy and Cauchy-Schwarz inequalities.
We play the trick of triangle inequality; more precisely,
for any $s\in [k\eta, (k+1)\eta)]$ where $0\leq k\leq [T/\eta]$ is an integer, we have, 
\begin{align*}
    &\nabla G(\beta_s)-\nabla\hat{G}(\beta_{[s/\eta]\eta}^{\eta}) \\
    = 
    &\underbrace{\nabla G(\beta_s)-\nabla G(\beta_s^{\eta})}_{\mathcal{I}_1}+
    \underbrace{\nabla G(\beta_s^{\eta})-\nabla G(\beta_{[s/\eta]\eta}^{\eta})}_{\mathcal{I}_2}
    +\underbrace{\nabla G(\beta_{[s/\eta]\eta}^{\eta})-
    \nabla\hat{G}(\beta_{[s/\eta]\eta}^{\eta})}_{\mathcal{I}_3},
\end{align*}
and for $\mathcal{J}$ we have,
\begin{align*}
    \Sigma_s-\hat{\Sigma}^{\eta}_{[s/\eta]\eta} = \underbrace{\Sigma_s-\Sigma_{k\eta}^{\eta}}_{\mathcal{J}_1}+
    \underbrace{\Sigma_{k\eta}^{\eta}-
    \hat{\Sigma}^{\eta}_{[s/\eta]\eta}}_{\mathcal{J}_2},
\end{align*}
where $\Sigma_{k\eta}^{\eta}$ is the continuous-time interpolation of $\beta^{\eta}(k)$ without noise and satisfies the following equation:
\begin{align*}
    \beta_t^{\eta} = \beta_0^{\eta} - \int_0^t\nabla G(\beta_{k\eta}^{\eta})ds
    +\int_0^t\Sigma^{\eta}_{k\eta} dW_s.
\end{align*}
We can estimtae $\mathcal{I}$ and $\mathcal{J}$ separately and the theorem can be proved.
\end{proof}

\section{Numerical Examples}\label{sec:numerical}
In this section, we will demonstrate the performance of the proposed methods, multi-variance replica exchange Langevin diffusion (m-reSGLD), by a sequence of numerical experiments. In section (\ref{exp1}), an inverse problem will be presented; and the experiments are designed to show that the proposed approaches are able to capture the multiple modes. In the rest of the numerical experiments sections, we will use the proposed methods as an optimization algorithm; in particular, we will consider several partial diffusion equations (PDEs) in the physics-informed neural network (PINN) framework.
\subsection{Experiment 1}
In this section, we will study an inverse problem which is commonly studied in the water resources research \cite{li2015adaptive}.
The problem can be solved by the sampling methods in the Bayesian framework. 
Compared to the traditional gradient based methods, the proposed methods converges faster and are able to capture the multiple modes \textcolor{black}{(multiple inverse quantity of interests, or, multiple modal posteriors from the Bayesian perspective)}. This is the consequence of the information exchange between the low temperature chain and high temperature chain.
\label{exp1}
We consider the following parabolic equation which models the contamination flows in the water:
\begin{align}
    & u_t = \nabla\cdot(\nabla u) +f, x\in \Omega, t\in[0, 0.03]\\
    & u(x, 0) = e^{-\|x-x_0\|^2/\alpha}\\
    & u(x, t) = v(x, t), x\in\partial\Omega.
\end{align}
In the above formulation, $\Omega = [0, 1]\times [0, 1]$; $\|\cdot\|$ is the Euclidean distance.
$\alpha = 2h^2$ and $\beta = M/(2\pi h^2)$ are two positive constants; where $h$ can be taken as the radius of the pollution source and $M$ is the strength of the initial contamination.
In the following two experiments, we set $h = 0.1$ and $M = 1$. 
The initial condition is partially given; more precisely, the pollution source $x_0\in \mathbb{R}^2$ is unknown. The target of the experiments is to trace back the initial pollution source given some observation data which are the concentration $u(x, t)$ at some sensors in the domain. 

We set the solution of the forward problem to be $u(x, t) = \beta e^{-(x-x_0)^2/\alpha}e^{-t}$. We will perform two experiments; the first experiment has 2 inverse solutions and the second one has infinite number of solutions. In both experiments, we apply the standard finite element solver. To be more specific, we are going to use a finer finite element solver with $\Delta x = 1/50$ in the low temperature chain which reinforces the local convergence; while in the high temperature chain which is easier to travel globally, we use a coarser finite element solver with $\Delta x = 1/25$. The time differences for both experiments are set to be $\Delta t = 0.0005$.

\subsubsection{Two inverse solutions}
Two sensors are placed at $s_1 = (0.5, 0.3)$ and $s_2 = (0.5, 0.6)$ and sensors measure the concentration (solution) of the pollution at the terminal time $t = 0.03$; in our experiment, $u(s_1, t) = u(s_2, t) = e^{-r^2/\alpha}\cdot\beta\cdot e^{-0.03}$, where $r = 0.2$. This setting guarantees to have 2 inverse solutions. 
In order to show the efficiency and accuracy of the proposed method, we will compare the method, namely the multi-variance replica exchange Langevin diffusion (m-reSGLD), with the single chain SGLD and the same energy function reSGLD.
That is, the first comparison test is the standard single chain SGLD; while 
we are going to use the same fine solver ($\Delta x = 1/50$) in both chains in the latter reSGLD method.

In Figure (\ref{fig:exp_2_pollution}), m-reSGLD and reSGLD are run with 30,000 iterations while SGLD is run with 60,000 iterations.
We can observe from the picture (\ref{fig:exp_2_pollution}) that the proposed method can capture both pollution sources while the single chain method fails to capture the other source. The reSGLD with the same solver (right picture) also can capture the multimodes as expected; in fact, this method should be the most accurate one; however, since we use the same fine solver in both chains, the computational cost is extremely high. In our tests, the proposed method runs around $8,713.7s$ while the fine solver method runs $14,964.3s$ on the computing resource. 
\begin{figure}[H]
\centering
\includegraphics[scale = 0.45]{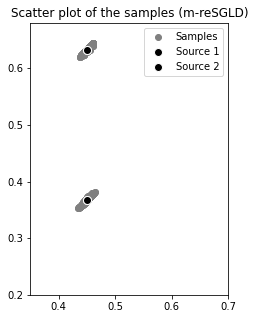}
\includegraphics[scale = 0.45]{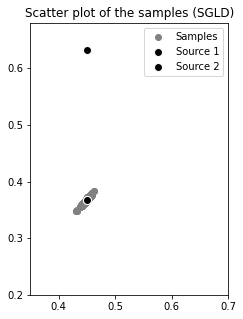}
\includegraphics[scale = 0.45]{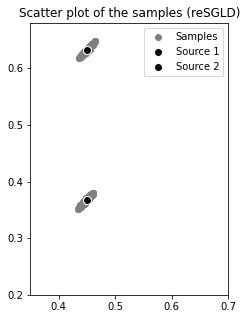}
\caption{Scatter plot of the proposed samples. We run 30,000 iterations for m-reSGLD and reSGLD while run 60,000 for the SGLD.
Left: m-reSGLD (proposed, we use two solvers in two chains); middle: single chain SGLD; right: reSGLD with the same solver in both chains. The black dots are the exact places of the pollution sources and the grey dots are the proposed samples. We can see that both m-reSGLD and reSGLD can capture \textcolor{black}{multiple modes of the inverse quantity of interest}; however, m-reSGLD takes $8,713.7s$ while reSGLD takes $14,964.3s$ to run 30,000 iterations. } 
\label{fig:exp_2_pollution}
\end{figure}

We also conduct the experiments using three methods but fixing the same computation time, i.e., the running time of all three methods are the same. We can see that the m-reSGLD can still capture the multiple modes; however, the reSGLD is unable to achieve the goal. The reason is:  
it takes a long time to evaluate 2 fine forward solvers (in both chains) and the computation terminates with not enough number of iterations.
The results are shown in Figure (\ref{fig:lin_exp_2_pollution}).
\begin{figure}[H]
\centering
\includegraphics[scale = 0.45]{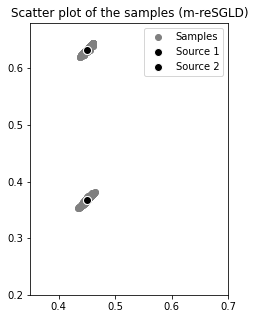}
\includegraphics[scale = 0.45]{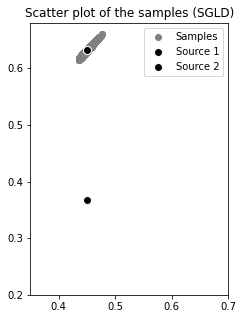}
\includegraphics[scale = 0.45]{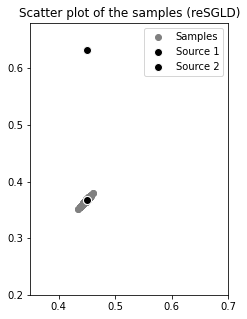}
\caption{Scatter plot of the proposed samples. The total computation time is fixed for all three methods. Left: m-reSGLD (2 solvers are used); middle: single chain SGLD; right: reSGLD with the same solver in both chains. The black dots are the exact places of the pollution sources and the grey dots are the proposed samples. The m-reSGLD can still capture \textcolor{black}{multiple modes of the inverse quantity of interest} however reSGLD fails due to the insufficient sampling iterations.} 
\label{fig:lin_exp_2_pollution}
\end{figure}

We want to comment on the accuracy of the solver used in the high temperature chain. There is a trade-off of the accuracy and the computational cost. More precisely, using a high accuracy solver can increase the accuracy; however, this will increase the computational cost and then limits the number of iterations of the algorithm. 

\subsubsection{Infinite inverse solutions}
To generate infinite number of solutions, we only place one sensor at position $(0.5, 0.3)$, we still measure the concentration at the terminal time and use the same value as before. It is not hard to see that: the inverse target is a circle centered at the sensor and the underlying problem has infinite number of solutions.
Same as the first experiment, we will compare our method with 2 other methods.

In the first set of experiments, which is shown in Figure (\ref{fig:exp_multi_pollution}), three algorithms are run with 60000 iterations. We can see that from the Figure (\ref{fig:exp_multi_pollution}) that, the proposed method greatly improves the sampling when compared to the single chain SGLD. The reSGLD method with the same energy method also works well; however, the computation time is around $25,832.8s$ while our proposed method only takes around $16,414.7s$ to finish generating the same number of samplings (iterations).
\begin{figure}[H]
\centering
\includegraphics[scale = 0.4]{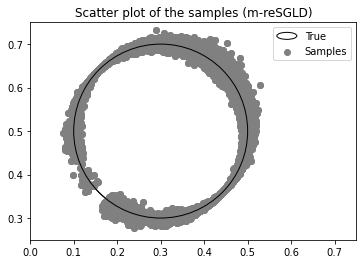}
\includegraphics[scale = 0.4]{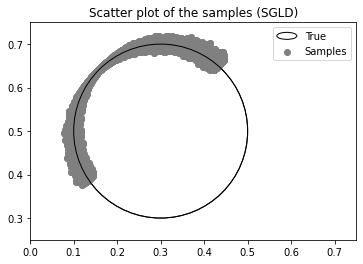}
\includegraphics[scale = 0.4]{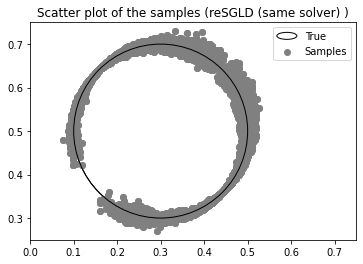}
\caption{Scatter plot of the proposed samples. Left: m-reSGLD; middle: single chain SGLD; right: reSGLD with the same solver in both chains. The black circle is the exact potential pollution sources and the grey dots are the proposed samples. We can see that both m-reSGLD and reSGLD can capture the multiple modes of inverse quantity of interests; however, m-reSGLD takes $16,414.7s$ while reSGLD takes $25,832.8s$ to run 30,000 iterations.} 
\label{fig:exp_multi_pollution}
\end{figure}

We also fix the computation time and compare the three methods, the results are demonstrated in Figure (\ref{fig:lin_exp_multi_pollution}).
We can see that m-reSGLD still performs well and is the best among the three methods.
In first set of experiments, the reSGLD is able to capture the multimodes; 
however, if the running time is the same with that of m-reSGLD, 
reSGLD performance is compromised.
This is due to the heavy computations of two fine forward solvers; as a result, training iterations are not enough.

\begin{figure}[H]
\centering
\includegraphics[scale = 0.4]{graphs/lin_replica_para_multi_half2550v2.png}
\includegraphics[scale = 0.4]{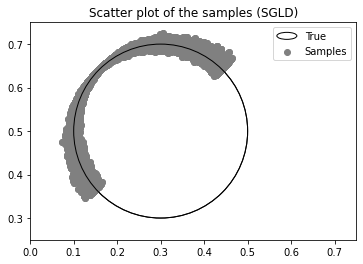}
\includegraphics[scale = 0.4]{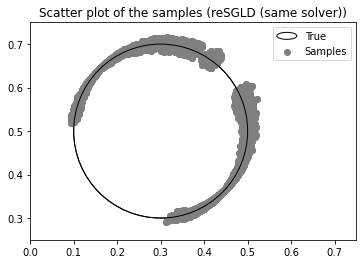}
\caption{Scatter plot of the proposed samples. Please note, the computation time is fixed same for three methods. Left: proposed method with the different solvers; middle: single chain SGLD; right: reSGLD with the same solver in both chains. The black circle is the exact potential pollution sources and the grey dots are the proposed samples. The m-reSGLD can still capture the multiple modes of inverse quantity of interests; however reSGLD fails due to the insufficient sampling iterations. } 
\label{fig:lin_exp_multi_pollution}
\end{figure}

\subsection{Bayesian Physics-informed neural network (PINN)}
Starting from this experiment, we are going to apply the proposed methods as an optimization approach in the neural network.
In particular, we will solve three PDEs in both forward and inverse settings with the physics-informed neural network (PINN) \cite{raissi2019physics}. 

 To learn the solutions for a given PDE with limited data, such as boundary/initial conditions and information about the source term, it is important to incorporate the available physical equations/laws in the design of deep neural networks. Physics-informed neural networks (PINN) \cite{liu2021deep, yang2021b, raissi2019physics}
 were proposed to realize the idea and have been successfully applied to many forward and inverse PDE applications subject to the law of physics that governs the data. Given a PDE of the form 
 \begin{align*}
     \mathcal{L}(u) &= f \text{ in } \omega \\
     \mathcal{B}(u) &= b \text{ on } \partial \omega
 \end{align*}
 where $\mathcal{L}$ is a differential operator and $\mathcal{B}$ is the boundary condition operator. $\omega$ is the computational domain and $f$ is the source term, $b$ is the boundary condition. Denote by $\mathcal{F}$ the neural network parameterized by $\beta$, the PINN aims to minimize the mean squared loss
 \begin{equation}
     \text{Loss}=\frac{w_1}{N_u} \sum_{i=1}^{N_u} |\mathcal{F}(x_i;\beta) - u(x_i)|^2 +  \frac{w_2}{N_f} \sum_{i=1}^{N_f} |\mathcal{L}(\mathcal{F}(x_i;\beta)) - f(x_i)|^2 + \frac{w_3}{N_b} \sum_{i=1}^{N_b} |\mathcal{B}(\mathcal{F}(x_i;\beta)) - b(x_i)|^2,
 \end{equation}
 where $w_1+w_2+w_3 = 1$ are the positive weights.
 
 However, {\color{black}{it was shown that the training of PINN is very challenging for PDEs, due to
significantly the multiple local optima in the loss function and varying gradients in the back-propagation \cite{paris_multiPinn}.}} Moreover, for the case when the data has noise, standard PINNs do not naturally have the capability to capture uncertainties. Recently, a Bayesian physics-informed neural network framework \cite{yang2021b} is proposed to deal with the issue. The Bayesian neural network differs from classical neural network by inducing a probability distribution on the model parameters of the network. This allows to estimate uncertainties in the predictions. Assume that the solution, source and boundary data has some noises, 
 \begin{equation*}
     \textbf{u}_{\epsilon} = \textbf{u} + \epsilon_{\textbf{u}}, \quad \quad \quad  \textbf{f}_{\epsilon} = \textbf{f} + \epsilon_{\textbf{f}},  \quad \quad \quad  \textbf{b}_{\epsilon} = \textbf{b} + \epsilon_{\textbf{b}}, 
 \end{equation*}
 where $\epsilon_{\textbf{u}}, \epsilon_{\textbf{f}}, \epsilon_{\textbf{b}}$ follow Gaussian distribution with zero mean, and standard deviation $\sigma_u, \sigma_f, \sigma_b$, respectively. Given data $\mathcal{D} = \mathcal{D}_u \cup  \mathcal{D}_f \cup  \mathcal{D}_b$, and $\mathcal{D}_u = \{\textbf{x}_{u}, \textbf{u}\}$, $\mathcal{D}_f = \{\textbf{x}_{f}, \textbf{f}\}$, $\mathcal{D}_b = \{\textbf{x}_{b}, \textbf{b}\}$,
 the likelihood is defined to be $p(\mathcal{D}|\beta)=  p(\mathcal{D}_u|\beta)p(\mathcal{D}_f|\beta)p(\mathcal{D}_b|\beta)$,
 \begin{align*}
    \mathcal{D}_u|\beta &\sim N(\mathcal{F}(x_u;\beta), \sigma_u),\\
    \mathcal{D}_f|\beta &\sim N(\mathcal{L}(\mathcal{F}(x_f;\beta)), \sigma_f),\\
    \mathcal{D}_b|\beta &\sim N(\mathcal{F}(x_b;\beta), \sigma_b).
 \end{align*}
The posterior distribution is then $p(\beta|\mathcal{D})\propto p(\mathcal{D}|\beta) p(\beta)$. For the posterior inference, the variational inference method and some variants of MCMC methods are commonly used.

 In this work, we adopt the replica exchange stochastic gradient MCMC approach. It stems from stochastic gradient Langevin dynamics (SGLD) , and is a simulated tempering approach which can accelerate the training and guarantees a better global point estimate.

\subsection{Experiment 3 (Quasi-gas dynamics (QGD))}
We solve a quasi-gas dynamics (QGD) model which is a kinetic equations under
the assumption that the distribution function is similar to a locally Maxwellian representation.
The QGD model is extensively studied and is not easy to be solved by the numerical methods \cite{chetverushkin2020computational, chung2021computational}; and we hence propose to solve it by the deep learning methods.
Consider the QGD equation in a polygonal domain $\Omega \subset \mathbb{R}^d$ ($d = 1$): 
\begin{eqnarray} \label{eqn:qgd}
\begin{split}
u_t + \alpha u_{tt} -  \nabla \cdot ( \kappa \nabla u) & =  f \quad &\text{in } (0,T] \times \Omega, \\
u|_{t=0} & = u_0 \quad &\text{in } \Omega, \\
u_t|_{t=0} & = v_0 \quad &\text{in } \Omega, \\
u & = 0 \quad &\text{on } \partial \Omega. 
\end{split}
\end{eqnarray}
Here, $T>0$ is the terminal time and $\alpha$ is a positive constant which measures the diffusion effect, $\kappa$ is conductivity of the model and is strictly positive, $f$ is the source of the equation with a proper regularity.
Furthermore, we assume that the initial conditions $u_0 \in H_0^1(\Omega)$ and $v_0 \in L^2(\Omega)$. 
In our experiments, 
$\Omega = [0, 1]$ and $T = 0.001$. 
We set the solution to be $u(x, t) = sin(2\pi x)e^{-t}$ and the initial conditions and source can be derived then.

In this work, $\alpha = 1$ and $\kappa = 1$. 
We design a four layer fully connected neural network 
($2\times 32\rightarrow 32\times 32 \rightarrow 32\times 32 \rightarrow 32\times 1$) activated by the $Tanh$ function.
$64$ points are uniformly placed in $[0, 1]$ in the low temperature chain 
and to accelerate the computation, $48$ points are used in the high temperature chain.
In both chains, we use $8$ points in time including the initial condition to discretize the time direction.

To demonstrate the performance of our method, we will compare the proposed method (m-reSGLD) with the single chain method, i.e., SGLD without the replica exchange.
\textcolor{black}{We will run the SGLD for two times: a low temperature SGLD (lt-SGLD) and a high temperature SGLD (ht-SGLD) respectively. 
To be more specific, the temperature in lt-SGLD is set to be the same as the low temperature in the m-reSGLD; meanwhile, the temperature in ht-SGLD is is set to be the same as the high temperature in the  m-reSGLD.}
The relative convergence is shown in Figure (\ref{fig:qgd_convergence}) and the mean and variance of the relative errors of the three methods are shown in Table (\ref{table_qgd}).
\begin{figure}[H]
\centering
\includegraphics[scale = 0.6]{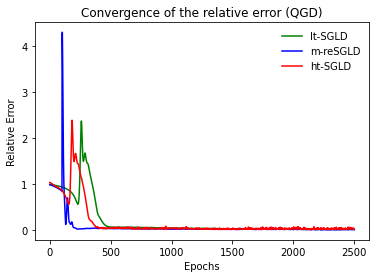}
\caption{Convergence of the relative error. The y axis is the relative error; x axis is the training epochs. 
Colors indicate the methods. Green (lt-SGLD): \textcolor{black}{SGLD with a temperature that is the same as the low temperature in the m-reSGLD}. Blue (m-reSGLD): proposed replica exchange method. Red (ht-SGLD): SGLD with a temperature that is the same as the high temperature in \textcolor{black}{the} m-reSGLD. } 
\label{fig:qgd_convergence}
\end{figure}

\begin{table}[H]
\centering
\begin{tabular}{||c c c||} 
\hline
 Methods & Mean of the relative error & Variance of the relative error \\ [0.5ex] 
\hline
\textcolor{black}{m-reSGLD}  & 0.01918168  & 2.324e-05  \\ [0.5ex]
\hline
lt-SGLD  & 0.02376848 & 0.00033714 \\ [0.5ex]
\hline
ht-SGLD  & 0.03477052 & 0.00014472  \\ [0.5ex]
\hline
\end{tabular}
\caption{QGD forward PINN. Mean and variance of the relative errors after burn-in.}
\label{table_qgd}
\end{table}
From Figure (\ref{fig:qgd_convergence}) and Table (\ref{table_qgd}), 
we can observe that the proposed method (blue curve) decays very fast, in particular when compared to the low temperature SGLD (green curve). 
The high temperature SGLD also has a quick convergence at the beginning since the high temperature encourages the global travel and the particle may move to the global optimal point earlier; however, the relative error of the ht-SGLD is the largest at the end and 
the blue curve (proposed m-reSGLD) stabilizes with the lowest relative error. 
Intuitively, this is due to the better local convergence of the low temperature SGLD.

We also plot the solution of the last time step in Figure (\ref{fig_qgd_sol}) as reference.
\begin{figure}[H]
\centering
\includegraphics[scale = 0.6]{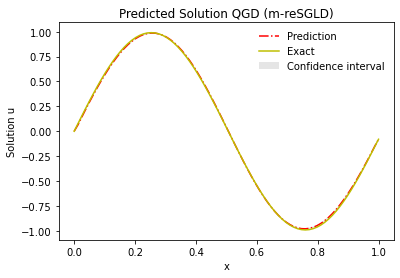}
\caption{\textcolor{black}{Exact solution at the terminal time (yellow curve) and confidence interval (grey band) of the predicted solution at the terminal time (red dashed curve). The prediction is calculated as the mean of the predicted solution of each epoch after the burn-in. The maximum prediction variance across the spatial direction $x$ is $0.011915$, the minimum is $0.005009$ and the mean is $0.00718$. The y axis is the solution at the terminal time and the x axis is the spatial domain of the problem.}} 
\label{fig_qgd_sol}
\end{figure}

\subsection{Experiment 4 (Inverse QGD)}
In this section, we are going to consider the QGD equation again but in the inverse setting, that is, $\alpha$ is not given and we will learn this constant together with the solution of the equation. The network is kept the same as before except that the last layer is of size $32\times 2$. We need additional observation information and hence use 10 true solution points equally spaced.
The  relative  convergence  is  shown  in  Figure (\ref{qgd_inv_convergence}) and the mean and variance of the relative errors of the three methods are shown in Table (\ref{table_qgd_inv_conv}).
\begin{figure}[H]
\centering
\includegraphics[scale = 0.6]{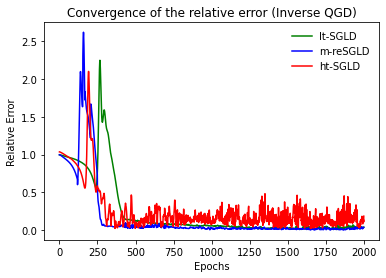}
\caption{Convergence of the relative error. The y axis is the relative error; x axis is the training epochs. 
Colors indicate the methods. Green (lt-SGLD): SGLD with a temperature that is the same as the low temperature in the m-reSGLD. Blue (m-reSGLD): proposed replica exchange method.
Red (ht-SGLD): SGLD with a temperature that is the same as the high temperature in the m-reSGLD. } 
\label{qgd_inv_convergence}
\end{figure}

\begin{table}[H]
\centering
\begin{tabular}{||c c c||} 
\hline
 Methods & Mean of the relative error & Variance  of the relative error \\ [0.5ex] 
\hline
\textcolor{black}{m-reSGLD}  & 0.02559266 & 0.00019022  \\ [0.5ex]
\hline
\textcolor{black}{lt-SGLD}  & 0.04922535 & 0.00071017  \\ [0.5ex]
\hline
\textcolor{black}{ht-SGLD}   & 0.15169148 & 0.00552021  \\ [0.5ex]
\hline
\end{tabular}
\caption{Inverse QGD PINN problem. Mean and variance of the relative errors after burn-in.}
\label{table_qgd_inv_conv}
\end{table}
From Table (\ref{table_qgd_inv_conv}) and Figure (\ref{qgd_inv_convergence}), we observe the similar results as before. The proposed method (m-reSGLD) demonstrates a faster convergence and stabilizes with a smaller relative error. We also plot the 
the solution of the last time step in Figure (\ref{fig_inv_qgd_sol}). 
\begin{figure}[H]
\centering
\includegraphics[scale = 0.6]{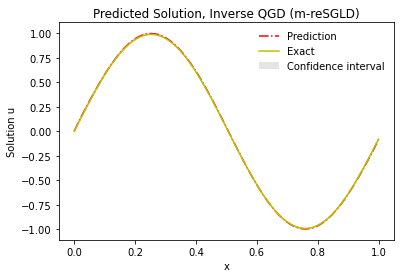}
\caption{Exact solution at the terminal time (yellow curve) and confidence interval (grey band) of the predicted solution at the terminal time (red dashed curve). The prediction is calculated as the mean of the predicted solution of each epoch after the burn-in. The maximum prediction variance across the spatial direction $x$ is $0.023186$, the minimum variance is $0.009243$ and the mean is $0.012495$.
The y axis is the solution at the terminal time and the x axis is the spatial domain of the problem.} 
\label{fig_inv_qgd_sol}
\end{figure}

Besides the convergence of the solution, we also present the prediction of the inverse target $\alpha$. Please check Table (\ref{table_qgd_inv}) for the details. We can observe that the proposed method gives us the best solution.
\begin{table}[H]
\centering
\begin{tabular}{||c c c c||} 
\hline
 Methods & True $\alpha$ & Mean of $\alpha$ & Variance of $\alpha$  \\ [0.5ex] 
\hline
\textcolor{black}{m-reSGLD} & 1.0 & 1.0033238 & 1.9756535e-05  \\ [0.5ex]
\hline
lt-SGLD & 1.0  & 1.0096047 & 0.0001136835  \\ [0.5ex]
\hline
ht-SGLD  & 1.0 & 0.9906514 & 0.0001133237  \\ [0.5ex]
\hline
\end{tabular}
\caption{Mean and variance of the inversion quantity $\alpha$ after burn in for the inversion QGD. The true $\alpha = 1$.}
\label{table_qgd_inv}
\end{table}

\subsection{Experiment 2 (Nonlinear Inverse)}
We consider an inverse PINN problem as fellow:
\begin{align}
    & -u_{xx}+\alpha u^2 = f, x\in \Omega = [-1, +1];\\
    & u(x) = u_0, x\in \partial \Omega,
\end{align}
where $u_0$ is a constant and will be derived according to the pre-set solution.
In this example, solution is set to be $u(x) = e^{-x^2/0.5}$ and the true $\alpha = 0.7$. 
The target is to learn the solution $u$ and find $\alpha$ at the same time.
We design a four-layer fully connected neural network 
($2\times 32\rightarrow 32\times 32 \rightarrow 32\times 32 \rightarrow 32\times 2$) activated by the $Tanh$ function.
$30$ points are uniformly placed for training in $[0, 1]$ in the low temperature chain 
and to accelerate the computation, $20$ points are used in the high temperature chain. 
For the observation, we place 5 sensors uniformly to measure the solutions of the equation.

In order to show the fast convergence and the low variance. We are going to compare our method (m-reSGLD) with the SGLD without exchange.
The temperatures of the single chain methods are set to be same with the low and high temperature in the proposed method respectively. More precisely,
in Figure (\ref{fig:inv_convergence}),
the red curve, \textcolor{black}{high-temperature SGLD (ht-SGLD)}, is the convergence of the relative error of the SGLD method with a temperature the same as the high temperature used in the proposed m-reSGLD (blue curve).
The green curve, \textcolor{black}{low-temperature SGLD (lt-SGLD)},
is the convergence of the relative error of the SGLD method with a temperature the same as the low temperature used in the proposed m-reSGLD.
The convergence of the method is demonstrated in Figure (\ref{fig:inv_convergence}); the mean and variance of the relative errors are shown in Table (\ref{table_inv}).
\begin{figure}[H]
\centering
\includegraphics[scale = 0.6]{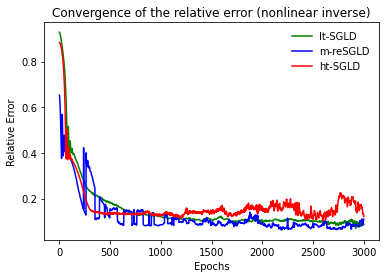}
\caption{Convergence of the relative error (Nonlinear Inverse). The y axis is the relative error; x axis is the training epochs. 
Colors indicate the methods. Green (lt-SGLD): SGLD with a temperature that is the same as the low temperature in the m-reSGLD. Blue (m-reSGLD): proposed replica exchange method. Red (ht-SGLD): SGLD with a temperature that is the same as the high temperature in the m-reSGLD. } 
\label{fig:inv_convergence}
\end{figure}

From the Figure (\ref{fig:inv_convergence}), we can observe that the blue curve (m-reSGLD) has a fast decay and converges to the lowest relative error. The fast decay can be intuitively explained: the high chain is exploring and once it is close to the optimal points, the likelihood then becomes large, which results in a bigger swapping rate (\ref{pp_rate}). \textcolor{black}{This mechanism will help the particle explore the whole loss function landscape and accelerate the convergence.}

\begin{table}[H]
\centering
\begin{tabular}{||c c c||} 
\hline
 Methods & Mean of relative error & Variance of relative error  \\ [0.5ex] 
\hline
m-reSGLD  & 0.07956814 & 8.962e-05  \\ [0.5ex]
\hline
lt-SGLD  & 0.09911878 & 5.032e-05  \\ [0.5ex]
\hline
ht-SGLD  & 0.15661082 & 0.00074168  \\ [0.5ex]
\hline
\end{tabular}
\caption{Mean and variance of the solution relative error after the burn in for the nonlinear inversion PINN.}
\label{table_inv}
\end{table}

The mean and variance of the predicted $\alpha$ are shown in Table (\ref{table_inv_alpha}).
\begin{table}[H]
\centering
\begin{tabular}{||c c c c||} 
\hline
 Methods & True $\alpha$ & Mean of $\alpha$ & Variance of $\alpha$ \\ [0.5ex] 
\hline
\textcolor{black}{m-reSGLD}& 0.7  & 0.6679792 &  0.001696  \\ [0.5ex]
\hline
lt-SGLD & 0.7 & 0.6406617 & 0.000534  \\ [0.5ex]
\hline
ht-SGLD & 0.7 & 0.8007492 & 0.002135 \\ [0.5ex]
\hline
\end{tabular}
\caption{Nonlinear inversion PINN. Mean and variance of the inversion quantity $\alpha$ after burn in for the inversion QGD. The true $\alpha = 0.7$.}
\label{table_inv_alpha}
\end{table}

The sample solutions are demonstrated in Figure (\ref{fig:inv_sol}).

\begin{figure}[H]
\centering
\includegraphics[scale = 0.35]{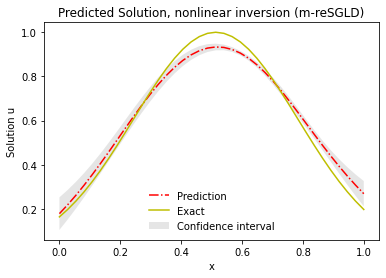}
\includegraphics[scale = 0.35]{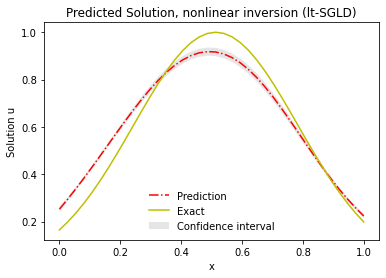}
\includegraphics[scale = 0.35]{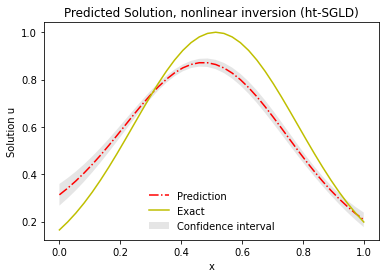}
\caption{Demonstration of the solutions for m-reSGLD (Left), lt-SGLD (middle) and ht-SGLD (right).
Exact solution : yellow curve. Confidence interval of the prediction: grey band. Predicted solution: red dashed curve. The prediction is calculated as the mean of the predicted solution of each epoch after the burn-in. The maximum prediction (m-reSGLD) variance across the spatial direction $x$ is $0.073663$, the minimum variance is $0.010018$ and the mean is $0.0307$. The y axis is the solution of the problem and the x axis is the spatial domain.}
\label{fig:inv_sol}
\end{figure}

\section{Conclusion}
The traditional replica exchange method assumes that the energy function estimates are the same; this means that in the Bayesian framework of solving inverse problems, one can only use the same solver in both chains; this results in a high computational cost.
In this work, we generalize the replica exchange method with different energy function estimates.
\textcolor{black}{To be more specific,
people can use a fine resolution solver in the low temperature chain while applying a coarse resolution solver in the high temperature chain.}
We hence use \textcolor{black}{forward solvers with different fidelities};
this setting can save a lot of computational cost in the high temperature chain and is much more efficient when compared to the \textcolor{black}{single} fidelity solver method while maintaining the accuracy.
We give an unbiased swapping rate and give an estimation of the discretization error. At the end, we verify that our method can capture \textcolor{black}{the multiple-mode non-Gaussian posterior distribution} in the Bayesian inverse problems; the examples demonstrate that our methods can beat the classical methods both in terms of the \textcolor{black}{efficiency} and accuracy.
We also apply the method as an sampling algorithm and solve several \textcolor{black}{forward and inverse} mathematical problems in the physics-informed \textcolor{black}{neural network} framework. 
Compared to the classical methods, our method reduces the relative error \textcolor{black}{of the solutions for the forward problems and the error of the inverse quantity of interests} for the inversion problems.

\bibliographystyle{abbrv}
\bibliography{references}

\begin{thebibliography}{10}

\bibitem{ahn2012bayesian}
S.~Ahn, A.~Korattikara, and M.~Welling.
\newblock Bayesian posterior sampling via stochastic gradient fisher scoring.
\newblock {\em arXiv preprint arXiv:1206.6380}, 2012.

\bibitem{sg-mcmc-convergence}
C.~Chen, N.~Ding, and L.~Carin.
\newblock On the convergence of stochastic gradient mcmc algorithms with
  high-order integrators.
\newblock {\em In Advances in Neural Information Processing Systems}, pages
  2278--2286, 2015.

\bibitem{chen2015convergence}
C.~Chen, N.~Ding, and L.~Carin.
\newblock On the convergence of stochastic gradient mcmc algorithms with
  high-order integrators.
\newblock In {\em Advances in Neural Information Processing Systems}, 2015.

\bibitem{chen2014stochastic}
T.~Chen, E.~Fox, and C.~Guestrin.
\newblock Stochastic gradient hamiltonian monte carlo.
\newblock In {\em International conference on machine learning}, pages
  1683--1691, 2014.

\bibitem{chen2020accelerating}
Y.~Chen, J.~Chen, J.~Dong, J.~Peng, and Z.~Wang.
\newblock Accelerating nonconvex learning via replica exchange langevin
  diffusion.
\newblock {\em arXiv preprint arXiv:2007.01990}, 2020.

\bibitem{chetverushkin2020computational}
B.~Chetverushkin, E.~Chung, Y.~Efendiev, S.-M. Pun, and Z.~Zhang.
\newblock Computational multiscale methods for quasi-gas dynamic equations.
\newblock {\em arXiv preprint arXiv:2009.00068}, 2020.

\bibitem{chung2020multi}
E.~Chung, Y.~Efendiev, W.~T. Leung, S.-M. Pun, and Z.~Zhang.
\newblock Multi-agent reinforcement learning accelerated mcmc on multiscale
  inversion problem.
\newblock {\em arXiv preprint arXiv:2011.08954}, 2020.

\bibitem{chung2021computational}
E.~Chung, Y.~Efendiev, S.-M. Pun, and Z.~Zhang.
\newblock Computational multiscale methods for parabolic wave approximations in
  heterogeneous media.
\newblock {\em arXiv preprint arXiv:2104.02283}, 2021.

\bibitem{dalalyan2017further}
A.~Dalalyan.
\newblock Further and stronger analogy between sampling and optimization:
  Langevin monte carlo and gradient descent.
\newblock In {\em Conference on Learning Theory}, pages 678--689. PMLR, 2017.

\bibitem{dauphin2014identifying}
Y.~N. Dauphin, R.~Pascanu, C.~Gulcehre, K.~Cho, S.~Ganguli, and Y.~Bengio.
\newblock Identifying and attacking the saddle point problem in
  high-dimensional non-convex optimization.
\newblock In {\em Advances in neural information processing systems}, pages
  2933--2941, 2014.

\bibitem{deng2020non}
W.~Deng, Q.~Feng, L.~Gao, F.~Liang, and G.~Lin.
\newblock Non-convex learning via replica exchange stochastic gradient mcmc.
\newblock In {\em International Conference on Machine Learning}, pages
  2474--2483. PMLR, 2020.

\bibitem{ding2014bayesian}
N.~Ding, Y.~Fang, R.~Babbush, C.~Chen, R.~Skeel, and H.~Neven.
\newblock Bayesian sampling using stochastic gradient thermostats.
\newblock In {\em Advances in Neural Information Processing Systems}, 2014.

\bibitem{efendiev2006preconditioning}
Y.~Efendiev, T.~Hou, and W.~Luo.
\newblock Preconditioning markov chain monte carlo simulations using
  coarse-scale models.
\newblock {\em SIAM Journal on Scientific Computing}, 28(2):776--803, 2006.

\bibitem{PSGLD}
C.~Li, C.~Chen, D.~Carlson, and L.~Carin.
\newblock Preconditioned stochastic gradient langevin dynamics for deep neural
  networks.
\newblock {\em In Thirtieth AAAI Conference on Artificial Intelligence}, 2016.

\bibitem{li2019diffusion}
Q.~Li and K.~Newton.
\newblock Diffusion equation-assisted markov chain monte carlo methods for the
  inverse radiative transfer equation.
\newblock {\em Entropy}, 21(3):291, 2019.

\bibitem{li2015adaptive}
W.~Li and G.~Lin.
\newblock An adaptive importance sampling algorithm for bayesian inversion with
  multimodal distributions.
\newblock {\em Journal of Computational Physics}, 294:173--190, 2015.

\bibitem{liu2021deep}
L.~Liu, T.~Zeng, and Z.~Zhang.
\newblock A deep neural network approach on solving the linear transport model
  under diffusive scaling.
\newblock {\em arXiv preprint arXiv:2102.12408}, 2021.

\bibitem{ma2015complete}
Y.-A. Ma, T.~Chen, and E.~Fox.
\newblock A complete recipe for stochastic gradient mcmc.
\newblock In {\em Advances in Neural Information Processing Systems}, pages
  2917--2925, 2015.

\bibitem{nguyen2019non}
T.~H. Nguyen, U.~Simsekli, and G.~Richard.
\newblock Non-asymptotic analysis of fractional langevin monte carlo for
  non-convex optimization.
\newblock In {\em International Conference on Machine Learning}, pages
  4810--4819. PMLR, 2019.

\bibitem{SGRLD}
S.~Patterson and Y.~W. Teh.
\newblock Stochastic gradient riemannian langevin dynamics on the probability
  simplex.
\newblock {\em In Advances in neural information processing systems}, pages
  3102--3110, 2013.

\bibitem{raginsky2017non}
M.~Raginsky, A.~Rakhlin, and M.~Telgarsky.
\newblock Non-convex learning via stochastic gradient langevin dynamics: a
  nonasymptotic analysis.
\newblock In {\em Conference on Learning Theory}, pages 1674--1703. PMLR, 2017.

\bibitem{raissi2019physics}
M.~Raissi, P.~Perdikaris, and G.~E. Karniadakis.
\newblock Physics-informed neural networks: A deep learning framework for
  solving forward and inverse problems involving nonlinear partial differential
  equations.
\newblock {\em Journal of Computational Physics}, 378:686--707, 2019.

\bibitem{roberts2002langevin}
G.~O. Roberts and O.~Stramer.
\newblock Langevin diffusions and metropolis-hastings algorithms.
\newblock {\em Methodology and computing in applied probability},
  4(4):337--357, 2002.

\bibitem{csimcsekli2017fractional}
U.~{\c{S}}im{\c{s}}ekli.
\newblock Fractional langevin monte carlo: Exploring l{\'e}vy driven stochastic
  differential equations for markov chain monte carlo.
\newblock In {\em International Conference on Machine Learning}, pages
  3200--3209. PMLR, 2017.

\bibitem{simsekli2016stochastic}
U.~Simsekli, R.~Badeau, T.~Cemgil, and G.~Richard.
\newblock Stochastic quasi-newton langevin monte carlo.
\newblock In {\em International Conference on Machine Learning}, volume~43,
  pages 642--651, 2016.

\bibitem{simsekli2020fractional}
U.~Simsekli, L.~Zhu, Y.~W. Teh, and M.~Gurbuzbalaban.
\newblock Fractional underdamped langevin dynamics: Retargeting sgd with
  momentum under heavy-tailed gradient noise.
\newblock In {\em International Conference on Machine Learning}, pages
  8970--8980. PMLR, 2020.

\bibitem{stuart2010inverse}
A.~M. Stuart.
\newblock Inverse problems: a bayesian perspective.
\newblock {\em Acta numerica}, 19:451--559, 2010.

\bibitem{teh2016consistency}
Y.~W. Teh, A.~H. Thiery, and S.~J. Vollmer.
\newblock Consistency and fluctuations for stochastic gradient langevin
  dynamics.
\newblock {\em Journal of Machine Learning Research}, 17, 2016.

\bibitem{vollmer2016exploration}
S.~J. Vollmer, K.~C. Zygalakis, and Y.~W. Teh.
\newblock Exploration of the (non-) asymptotic bias and variance of stochastic
  gradient langevin dynamics.
\newblock {\em The Journal of Machine Learning Research}, 17(1):5504--5548,
  2016.

\bibitem{paris_multiPinn}
S.~Wang, H.~Wang, and P.~Perdikaris.
\newblock On the eigenvector bias of fourier feature networks: From regression
  to solving multi-scale pdes with physics-informed neural networks.
\newblock {\em Computer Methods in Applied Mechanics and Engineering},
  384:113938, 2021.

\bibitem{hessian-sgld-sa}
Y.~Wang, W.~Deng, and G.~Lin.
\newblock An adaptive hessian approximated stochastic gradient mcmc method.
\newblock {\em Journal of Computational Physics}, 432:110150, 2021.

\bibitem{psgld-sa}
Y.~Wang, W.~Deng, and G.~Lin.
\newblock Bayesian sparse learning with preconditioned stochastic gradient mcmc
  and its applications.
\newblock {\em Journal of Computational Physics}, 432:110134, 2021.

\bibitem{sgld}
M.~Welling and Y.~W. Teh.
\newblock Bayesian learning via stochastic gradient langevin dynamicsn.
\newblock {\em In Proceedings of the 28th international conference on machine
  learning (ICML-11)}, pages 681--688, 2011.

\bibitem{xu2017global}
P.~Xu, J.~Chen, D.~Zou, and Q.~Gu.
\newblock Global convergence of langevin dynamics based algorithms for
  nonconvex optimization.
\newblock {\em arXiv preprint arXiv:1707.06618}, 2017.

\bibitem{yang2021b}
L.~Yang, X.~Meng, and G.~E. Karniadakis.
\newblock B-pinns: Bayesian physics-informed neural networks for forward and
  inverse pde problems with noisy data.
\newblock {\em Journal of Computational Physics}, 425:109913, 2021.

\bibitem{zhang2017hitting}
Y.~Zhang, P.~Liang, and M.~Charikar.
\newblock A hitting time analysis of stochastic gradient langevin dynamics.
\newblock In {\em Conference on Learning Theory}, pages 1980--2022. PMLR, 2017.

\end{thebibliography}
\end{document}